\documentclass[11pt]{article}
\usepackage{amsmath,amssymb,amsthm,amscd,dsfont}
\usepackage{amsfonts,latexsym,rawfonts,amsmath,amssymb,amsthm}
\usepackage{lscape}
\usepackage{amscd, float,times,rotating}
\usepackage{pb-diagram}
\usepackage{hyperref}
\usepackage[pagewise]{lineno}

\numberwithin{equation}{section}

\def\M{{\partial M}}

\newtheorem{prop}{Proposition}[section]
\newtheorem{theo}[prop]{Theorem}
\newtheorem{lemm}[prop]{Lemma}

\newtheorem{rema}[prop]{Remark}

\def\begeq{\begin{equation}}
\def\endeq{\end{equation}}

\begin{document}
\title{Boundary Estimate of Asymptotically Hyperbolic Einstein Manifolds of Even Dimension}
\author{Xiaoshang Jin}
\date{}
\maketitle
\begin{abstract}
In this paper, we study the finite boundary regularity and estimates of an asymptotically hyperbolic Einstein manifold in even dimension $n+1.$ We show that if the initial compactification is $C^{n-1}$ and the $(n-3)$-th derivative of its scalar curvature is H\"older continuous, then the AHE metric is $C^{m,\alpha}$ conformally compact provided the boundary metric is $C^{m,\alpha}$. This is an improvement of Helliwell's result. We also provide an estimate of the Yamabe compactification metric in the new structure.
\end{abstract}
\section{Introduction}  \label{sect1}
\par It is well-known that there is a relationship between an asymptotically hyperbolic Einstein metric on a smooth manifold and the conformal structure on the boundary. The research has become the main theme in the study of conformal geometry since the introduction of the AdS/CFT correspondence in the quantum theory of gravity in theoretic physics by Maldacena in \cite{maldacena1999large}.
\par Let $\overline{M}$ be a compact smooth manifold of dimension $n+1$ with non-empty boundary $\M$ and $M$ be its interior. A complete metric $g^+$ on $M$ is called (smoothly) conformally compact if there exits a defining function $\rho$ on $\overline{M}$ such that the conformally equivalent metric $$g=\rho^2g^+$$ can extend to a smooth Riemannian metric on $\overline{M}.$ The defining function is smooth on $\overline{M}$ and satisfies
 \begin{equation}\label{1.1}
 \left\{
    \begin{array}{l}
    \rho>0\ \ in \ M
    \\\rho=0\ \ on\ \M
    \\d\rho\neq 0\ \ on\  \M
    \end{array}
 \right.
 \end{equation}
 We could also define the $C^{m,\alpha}$ ($W^{k,p}$) conformally compact if $\overline{M}$ is a compact $C^{m+1,\alpha}$ ($W^{k+1,p}$) manifold and there is a $C^{m+1,\alpha}$ ($W^{k+1,p}$) defining function $\rho$ such that $g=\rho^2g^+$ can extend to a $C^{m,\alpha}$ ($W^{k,p}$) Riemannian metric on $\overline{M}.$ Here $C^{m,\alpha}$ and $W^{k,p}$ are usual H\"older space and the Sobolev space.
 \par Given a conformally compact metric $g^+$ and its compactification $g,$ The induced metric $h=g|_{\M}$  is called the boundary metric associated to the compactification $g.$ It is easy to see that $(M,g^+)$ carries a well-defined conformal structure on the boundary $\M$ as each metric $\hat{g}\in [h]$ is induced by $\hat{g}=\hat{\rho}^2g^+$ for a defining function $\hat{\rho}.$ Then  the conformal class $[h]$ is uniquely determined by $(M,g^+).$ We call $[h]$ the conformal infinity of $g^+.$
 \par A conformally compact manifold is said to be asymptotically hyperbolic (AH) if its sectional curvature goes to $-1$ when
approaching the boundary at infinity $(\rho\rightarrow 0).$
\par If in addition, $g^+$ is Einstein, i.e.
  \begin{equation}\label{1.2}
Ric_{g^+}+ng^+=0,
 \end{equation}
 then we say $(M,g^+)$ is a conformally compact Einstein manifold.
By a direct calculation, a $C^2$ conformally compact Einstein manifold $(M,g^+)$ is asymptotically hyperbolic. Hence we can also say
that $(M,g^+)$ is an  asymptotically hyperbolic Einstein manifold or AHE manifold.
\\
\par To better understand the AHE manifold, we need to make good use of Einstein equation \ref{1.2} carefully. There are two approaches so far. The first approach is to avoid using degenerate elliptic PDE. We can consider the Bach flat equation in dimension 4 and ambient obstruction tensor flat equation in higher even dimension. If we choose the harmonic coordinates near the boundary and select some special compactification, such as Yamabe compactification or Fefferman-Graham compactification \cite{chang2018compactness1}\cite{chang2020compactness}, then we get the
$(n+1)$-th elliptic equation on metric $g$ with some complicated boundary conditions of
higher order of derivatives on metric and curvature. The second approach is introduced by John Lee in \cite{lee2006fredholm}.
He described the special "boundary M\"obius coordinate charts" on an asymptotically hyperbolic manifold that
relate the geometry to that of hyperbolic space. He also introduced the
weighted Sobolev and H\"older spaces of AH manifols. Then he studied some geometric elliptic operators
on these spaces. John Lee showed an existence result of AHE metric in \cite{lee2006fredholm}.  In \cite{chrusciel2005boundary}, Chru\'sciel et al. constructed the harmonic diffeomorphism near the infinity to obtain
a good structure near boundary where Einstein equation could be written as a degenerate elliptic PDE on metric $g^+$ of second order.
That is so-called ’gauged Einstein equation’. Then with the help of the properties of the weighted H\"older spaces of AH manifolds, they use polyhomogeneity result of some specific degenerate equation to obtain a good result of the boundary regularity.
\\
\par The regularity problem was first raised by Fefferman and Graham in \cite{graham1985conformal}. Given an AHE manifold $(M,g^+)$ and its compactification $g=\rho^2g^+,$ suppose that $g$ is $C^{k,\sigma}$ smooth to the boundary and the induced metric $h=g|_{\M}$ is smooth on the boundary. Is it true that the $g^+$ has a compactification of better regularity because of the Einstein equation? We would explain the problem more clearly from the following two aspects.
\par 1. We cannot expect that the conformal compactification of
an arbitrary conformally compact Einstein metric will necessarily have
optimal regularity for all smooth structure of $\overline{M}.$ Hence we need to find the best coordinate charts at infinity. It is also necessary to study the regularity of the new structure.
\par 2. Is there a minimal initial regularity of a conformal compactification $g$ that is needed to any possible improvement? This problem may have some relations with the accurate definition of AHE metric.
\par In \cite{chrusciel2005boundary}, Chru\'sciel et al. proved that if the initial conformal compactification $g$ is $C^2$ and the boundary metric is smooth, then $g^+$ has a conformal compactification that is smooth up to the boundary in the sense of $C^{1,\lambda}$ diffeomorphism in dimension 3 and all even dimensions, and polyhomogeneous smooth in odd dimensions greater than 3. However, their result only hold for smooth case. It is believed that if we use their method to deal with the finite regularity problem, we may loss half regularity in this situation.
\par When dealing with the finite regularity problem, we could use the first approach as mentioned above. Anderson claims in \cite{anderson2008einstein} that a version of these results holds for $n+1 = 4$ when the initial conformal compactification of $g^+$ is $W^{2,p}$ for some $p > 4.$ I am unable to verify the theorem under this condition. As a supplementary proof, in \cite{jin2019finite}, we prove his conclusion where we assume that the initial compactification $g$ is $C^2$ and the scalar curvature is H\"older continuous.
\par  In \cite{helliwell2008boundary}, Helliwell considered the Fefferman-Graham ambient obstruction tensor instead of Bach tensor in higher even dimensions. He assumed the initial compactification $g$ is at least in $C^{n,\alpha}$ for a $(n+1)-$ smooth manifold then obtained the finite regularity results. The project we work on in this paper is to extend his results by weakened the regularity of the initial compactification.
\par In \cite{chang2021compactness}, Chang-Ge-Jin-Qing studied a special compactification which they called "adopted metric". The metric was defined by solving a PDE on $M.$ They showed that (lemma 3.1), if  $n+1\geq 6$ is even and the $n+1-$ dimensional adopted metric $g^*$ is a $C^{n-1}$ compactification of an AHE metric $g^+,$  then $g^*$ is $C^{m,\alpha}$ when the boundary metric is $C^{m,\alpha}.$ They also get the $C^{m,\alpha}$ estimate of $g^*$ near the boundary in harmonic charts.
\\
\par This is the main result:
\begin{theo}\label{theorem 1.1}
Let $(M,g^+)$ be an asymptotically hyperbolic Einstein manifold of even dimension $n+1$ and $g=\rho^2g^+$ be a $C^{n-1}$ compactification. If the scalar curvature $S_g\in C^{n-3,\sigma}(\overline{M})$ for some $\sigma>0,$ the boundary metric $h=g|_{\partial M}\in C^{m,\alpha}(\partial M)$ with $m\geq n-1,\alpha\in(0,1),$ then under a $C^{n-1,\lambda}$ coordinates change, $g^+$ has a $C^{m,\alpha}$ conformally compactification $\tilde{g}=\tilde{\rho}^2g^+$ with the boundary metric $\tilde{g}|_{\M}=h.$ Furthermore, The new coordinates form $C^{m+1,\alpha}$ differential structure of $\overline{M}$ and
$\tilde{\rho}$ is a $C^{m+1,\alpha}$  defining function in the new structure.
\end{theo}
It is clear that for an AHE metric $g^+$ and $g=\rho^2g^+,$ we have that $|\nabla\rho|_{\rho^2g^+}=1$ on $\M.$ In fact, if we fix a representative $h\in [h],$ then there is a unique defining function which we call geodesic defining function $r$ such that $|\nabla r|_{r^2g^+}\equiv 1$ in a neighbourhood of $\M.$ (See \cite{lee1994spectrum} for more details.) It implies that
in $\M\times(0,\epsilon)\subset M,$ $g^+$ has the normal form
$$g^+=r^{-2}(dr^2+g_r),$$
where $g_r$ is a 1-parameter family of metrics on $\M.$ Then we have the following
expansions of the metric \cite{graham1985conformal}
\begin{equation}\label{1.3}
g_r=h+g^{(2)}r^2+(even\ powers\ of\ r) + g^{(n-1)}r^{n-1}+g^{(n)}r^n+\cdots
\end{equation}
 if $\dim M= n+1$ is  even.
 \begin{equation}\label{1.4}
g_r=h+g^{(2)}r^2+(even\ powers\ of\ r) + g^{(n)}r^{n}+fr^n\log r+\cdots
\end{equation}
if $\dim M= n+1$ is odd.
\\ Here $g^{(2i)}$ are determined by $h$ for $2i < n$ and $g^{(n)}$ is non-local and determined by the $g^+$ and $h.$ The rest of power series is determined by $h$ and $g^{(n)}$ when $n+1$ is even. In other words, if we have some information of $h$ and $g^{(n)},$ we could obtain the properties of $g,$ such as the regularity and compactness of $g.$  In Helliwell’s paper, that he needed the initial compactification  to have $C^{n,\sigma}$ regularity seems very natural. Now we reduce it to
$C^{n-1}$ and it is a big step as we don't use any information about the non-local term.
\\
\par We also get an estimate of the Yamabe compactification locally in the new coordinates, that is:
\begin{theo}\label{theorem 1.2}
Suppose that $(M,g^+)$ is an AH Einstein manifold of even dimension $n+1$ and $g=\rho^2g^+$ is a $C^{n-1}$ Yamabe compactification. The boundary metric $h=g|_{\partial M}\in C^{m,\alpha}(\partial M)$ with $m\geq n-1,\alpha\in(0,1).$ For any
$p\in\M,$ let $(\overline{U},\{x^\beta\}_{\beta=0}^n)$ be a local harmonic chart of $p$ and $D=\overline{U}\cap\M.$  Then we have
\begin{equation}\label{1.5}
  |g|_{C^{m,\alpha}(\overline{U})}\leq C
\end{equation}
where $C$ depends on $n,\overline{U},|g|_{C^{1,\sigma}(\overline{U})},|Rm|_{C^{n-3}(\overline{U})},|h|_{C^{m,\alpha}(D)},$ the minimum eigenvalue of $g$ in $\overline{U}$ and $|\rho|_{C^{n-1,\sigma}(\overline{U})}$ for some $\sigma\in (0,1).$
\end{theo}
\begin{rema}
From Theorem 1.2, we can deduce that for the $C^{n-1}$ Yamabe compactification $g=\rho^2g^+,$ if the $C^{1,\sigma}$ harmonic
radius $r_0$ (see section 3 in \cite{anderson2004boundary}) and $|Rm_g|_{C^{n-3}}$ are bounded, then $(\M\times [0,r_0),g)$ are $C^{m,\alpha}$ compact in Cheeger-Gromov topology when the boundary metric $(\M,h)$ are $C^{m,\alpha}$ compact.
\end{rema}

\par This is the outline of the paper. In section \ref{sect2}, we introduce some basic facts about asymptotically hyperbolic metrics and some useful tools. We show that the Yamabe compactification always exists and the Yamabe compactification is still $C^{n-1}$ under the conditions in Theorem \ref{theorem 1.1}. Then we introduce the harmonic coordinates near boundary. This is a new structure of $\overline{M}$ and is only of $C^{n-1,\alpha}$ regularity.  We also consider the ambient obstruction tensor for an even dimensional manifold and give a representation in local coordinates.  Lastly,  we present the intermediate Schauder theory, i.e. the $C^\alpha$ and $C^{1,\alpha}$ estimates for elliptic PDE of order 2 under weak regularity hypotheses.
\par In section \ref{sect3}, we deduce some boundary conditions, including the Dirichlet conditions of metric and Ricci curvature, the  Neumann conditions of Ricci curvature and the oblique derivative conditions of metric if the compactification $g$ is only $C^2.$ We also give the boundary conditions of higher order of derivatives on Ricci curvature when the compactification $g\in C^{n-1}.$
\par In section \ref{sect4}, we prove the main theorem. Firstly, we study the term of fewer derivatives in the ambient obstruction tensor and prove that it contains at most $(n-3)$-th derivative on curvature. Then we prove the main theorems with the intermediate Schauder theory. In the end, we use the ambient obstruction tensor flat equation to prove the regularity of the new structure and the new defining function.

\section{Preliminaries}  \label{sect2}
Let $(M,g^+)$ be an $(n+1)$-dimensional conformally compact Einstein manifold and $g=\rho^2g^+$ is a compactification. Then
\begin{equation}\label{2.1}
K_{ab}=\frac{K^+_{ab}+|\nabla\rho|^2}{\rho^2}-\frac{1}{\rho}[D^2\rho(e_a,e_a)+D^2\rho(e_b,e_b)],
\end{equation}
\begin{equation}\label{2.2}
Ric=-(n-1)\frac{D^2\rho}{\rho}+[\frac{n(|\nabla\rho|^2-1)}{\rho^2}-\frac{\Delta\rho}{\rho}]g,
\end{equation}
\begin{equation}\label{2.3}
S=-2n\frac{\Delta\rho}{\rho}+n(n+1)\frac{|\nabla\rho|^2-1}{\rho^2}.
\end{equation}
Here $K_{ab},Ric,S$ are the sectional curvature, Ricci curvature and scalar curvature of $g$ and $D^2$ denote the Hessian. (Readers can see \cite{besse2007einstein} for the conformal transformation law of curvatures.)
\par If $g$ is a $C^2$ compactification, then from (\ref{2.3}),$|\nabla\rho|=1$ on $\M.$  Hence (\ref{2.1}) implies that $K^+_{ab}$ tends to $-1$ as $\rho\rightarrow 0.$ Hence a $C^2$ conformally compact Einstein manifold is asymptotically hyperbolic.
Let $D^2\rho|_{\M}=A$ denote the second fundamental form of $\M$ in $(\overline{M},g).$
The equation (\ref{2.2}) further implies that $\M$ is umbilic.
\subsection{Yamabe Compactification and Harmonic Coordinates}
\begin{lemm}
Let $(M,g^+)$ be a conformally compact n-manifold with a $W^{2,p}$ conformal compactification $g=\rho^2g^+$ where $p> n/2.$ Suppose that $h=g|_{\M}$ is the boundary metric. Then there exits a $W^{2,p}$ constant scalar curvature compactification $\tilde{g}=\tilde{\rho}^2g^+$ with boundary metric $h.$
\end{lemm}
This is Lemma 2.1 in \cite{jin2019finite}. We call $\tilde{g}$ the Yamabe compactification of $g^+.$ Furthermore, if $g\in C^{n-1}$ and $S_g\in C^{n-3,\sigma}$ for some $\sigma>0,$ we could prove that the new defining function $\tilde{\rho}\in C^{n-1,\sigma}$ and $\tilde{g}=\tilde{\rho}^2g$ is still $C^{n-1}.$
In the following of this paper, we don't distinguish $g$ with $\tilde{g}.$ When we refer to the compactification $g,$ we mean that
 the scalar curvature of $g$ is constant $(-1)$ near the boundary and the defining function is $C^{n-1,\sigma}.$
\\
\par In the rest of the paper, if there are no special instructions, any use of indices will
follow the convention that Roman indices will range from 1 to n, while Greek indices range from 0 to n.
\par We call the coordinates $\{x^\beta\}_{\beta=0}^n$ harmonic coordinates with respect to $g$ if $$\Delta_gx^\beta=0$$
 for $0\leq\beta\leq n.$ For any point $p\in\M,$ let$(\overline{V},\{y^\beta\}_{\beta=0}^n)$ be its local chart satisfying that $y^0|_{\M}=0.$ If $g$ is Lipschitz, then there exists a harmonic chart $(\overline{U},\{x^\beta\}_{\beta=0}^n)$ around $p$ and $x^0|_{\M}=0.$ (See \cite{jin2017extension} or \cite{jin2019finite}.)
 \par In particular, if $g\in C^{n-1},$ then these two charts are $C^{n-1,\alpha}$ compatible in the sense that $x\in C^{n-1,\alpha}(y)$ for any $\alpha\in(0,1).$ Hence
$$g_{\alpha\beta}=g(\frac{\partial}{\partial x^\alpha},\frac{\partial}{\partial x^\beta})\in C^{n-2,\alpha}(\overline{M})$$
and  $\bigcup\limits_{p\in\M} (\overline{U}_p,\{x^\beta\}_{\beta=0}^n)$ forms a $C^{n-1,\alpha}$ structure of $\M\times[0,\epsilon).$
\\ In harmonic coordinates $\{x^\beta\}_{\beta=0}^n$, the Ricci tensor could be written as:
$$\Delta g_{ij}=-2R_{ij}+Q(g,\partial g)$$
where $Q(g,\partial g)$ is a polynomial of $g$ and $\partial g.$ One can see \cite{deturck1981some} for more details.

\subsection{The Ambient Obstruction Tensor Flat Equation}
In \cite{graham1985conformal}, Fefferman and Graham gave the introduction of Ambient Obstruction Tensor $\mathcal{O}_{ij}$ for an even dimensional manifold. This is a
generalization of the Bach tensor in dimension 4. It involves $n+1$ derivatives of the metric on a
manifold of even dimension $n+1\geq 4$ and it is a conformal invariant and vanishes for Einstein metric.
It follows from Theorem 2.1 in \cite{graham2005ambient} that in local coordinates, we have the following obstruction tensor flat equation:
\begin{equation}\label{2.4}
\mathcal{O}_{ij}=\Delta^{\frac{n+1}{2}-2}(P_{ij,k}^{\ \ \ \ k}-P_{k\ \ ,ij}^{\ k})+lots=0,
\end{equation}
where  $P_{ij}=\frac{1}{n-1}(R_{ij}-\frac{S}{2n}g_{ij})$ is the Schouten tensor and lots denotes quadratic and higher terms in curvature involving fewer derivatives. We will study  the ambient obstruction tensor in section 4.

\subsection{Intermediate Schauder Estimate}
Since the initial compactification $g$ is only $C^{n-1}$ and the ambient obstruction flat equation involves $n + 1$ derivatives of the metric, we can't use the classical Schauder theory directly. Hence we will firstly introduce the conception of "intermediate Schauder theory" of elliptic PDE in \cite{gilbarg1980intermediate}, i.e. $C^\alpha$ and $C^{1,\alpha}$ estimate.
\par Let $\Omega$ be a bounded convex domain in $\mathds{R}^n$ and $a$ be a positive number satisfying  $a=k+\beta$ ($k\in\mathds{N}, \beta\in(0,1]$.)
Defining
$$|u|_a=\sum\limits_{|\alpha|\leq k}|D^\alpha u|_0+\sum\limits_{|\alpha|=k}\sup\limits_{x,y\in \Omega}\frac{|D^\alpha u(x)-D^\alpha u(u)|}{|x-y|^\beta}.$$
Let $H_a(\Omega)$ denote the H\"older space of functions with finite norm $|u|_a$ on $\Omega,$ i.e. $H_a(\Omega)=C^ {k,\beta}(\overline{\Omega})$. Setting
$$\Omega_\delta=\{x\in\Omega|dist(x,\partial\Omega)>\delta\}$$
Let $b$ be a number satisfying $a+b\geq 0$ and define
$$|u|_{a,\Omega}^{(b)}=|u|_{a}^{(b)}=\sup\limits_{\delta>0}\delta^{a+b}|u|_{a,\Omega_\delta}$$
Let $H_a^{(b)}(\Omega)$ denote the space of functions $u$ in $H_a(\Omega_\delta),(\forall \delta>0)$ such that $|u|_{a,\Omega}^{(b)}$ is finite.
Let $H_a^{(b-0)}(\Omega)$ be the space of functions $u$ in $H_a^{(b)}(\Omega)$ such that if $\delta\rightarrow 0,$ then $\delta^{a+b}|u|_{a,\Omega_\delta}\rightarrow 0$.
\par Basic properties: (the following constant $C$ depends on $a,b,\Omega.$)
\begin{itemize}
\item[1.] $H_a^{(-a)}(\Omega)=H_a(\Omega)=C^{k,\beta}(\overline{\Omega})$. Noticing that if $a$ is positive integer, $H_a(\Omega)=C^{a-1,1}(\overline{\Omega});$
\item[2.] If $b\geq b'$, then $|u|_{a,\Omega}^{(b)}\leq C|u|_{a,\Omega}^{(b')};$
\item[3.] If $0\leq a'\leq a,a'+b\geq 0$ and $b$ is not a non-positive integer, then $|u|_{a',\Omega}^{(b)}\leq C|u|_{a,\Omega}^{(b)}$;
\item[4.] If $0\leq c_j\leq a+b, a\geq 0,j=1,2,$ then
$$|uv|_{a}^{(b)}\leq C(|u|_{a}^{(b-c_1)}|v|_{0}^{(c_1)}+|u|_{0}^{(c_2)}|v|_{a}^{(b-c_2)})$$
Specially, if $u$ and $v$ are continuous functions (bounded), then $|uv|_{a}^{(b)}\leq C(|u|_{a}^{(b)}+|v|_{a}^{(b)})$. Here $C$ also depends on
the $L^\infty$ norm of $u$ and $v.$
\end{itemize}
With the preparations above, we could state the intermediate Schauder theory. Assuming that $\Omega$ is a bounded $C^\gamma$ domain where $\gamma\geq 1$ and $a,b$ are not integer satisfying
$$0<b\leq a, \ \ a>2, \ \ b\leq\gamma$$
Let
$$P=\sum\limits_{|\alpha|\leq 2}p_\alpha(x)D^\alpha$$
be the elliptic differential operator of second order on $\overline{\Omega}$  where
$$p_\alpha\in H_{a-2}^{(2-b)}(\Omega), \ \ if \ |\alpha|\leq 2$$
$$p_\alpha\in H_0(\Omega),\ \ if \ |\alpha|= 2$$
$$p_\alpha\in H_{a-2}^{(2-|\alpha|-0)}(\Omega), \ \ if \ b<|\alpha|.$$
Then if $u\in C^0(\overline{\Omega})\cap C^2(\Omega)$ is a solution of
\begin{equation}\label{2.5}
Pu=f \ \ in \ \Omega, \ \ \ u=\varphi \ \ on \ \partial\Omega
\end{equation}
where $f\in H_{a-2}^{(2-b)}(\Omega), \varphi\in H_b(\partial\Omega,$ it follows that $u\in H_{a}^{(-b)}(\Omega)$ and satisfies an estimate
$$u_{a}^{(-b)}\leq C(|u|_0+|\varphi|_{b,\partial\Omega}+|f|_{a-2}^{(2-b)}).$$
Here $C$ depends on $\Omega,a,b$ the norms of the coefficients and their minimum eigenvalue. If $p_0\leq 0,$ the Dirichlet problem \ref{2.5} has a unique solution in $H_{a}^{(-b)}(\Omega)$ and a corresponding Fredholm-type theorem holds in general.
\section{The Boundary Conditions}\label{sect3}
\par In this section, we derive a boundary problem for $g$ and Ricci curvature of a conformal compact Einstein manifold in the harmonic coordinates as defined in section 2. For any $p\in \M,$ there is a neighborhood $\overline{U}$ contains $p$ and a local harmonic chart $\{x^\beta\}.$  Let $D=\overline{U}\cap \M$ be the boundary portion and let $g\in C^{n-1}(\overline{U})$ be the Yamabe compactification. We will firstly give the Dirichlet and Neumann boundary conditions of $g$ and $Ric(g)$ on D. Here we state that the boundary conditions in this section hold for all dimension.
\par Firstly, as it is showed in \cite{helliwell2008boundary} and \cite{jin2019finite} that, if $g^+$ is $C^2$ conformally compact, we have the following boundary conditions:
\begin{prop}
Let $(M,g^+)$ be an $(n+1)$-dimensional conformally compact Einstein manifold with a $C^{3,\alpha}$ Yamabe compactification $g=\rho^2g^+.$ $g|_{\M}=h$ is the boundary metric. Suppose that $\{x^\beta\}_{\beta=0}^n$ are any coordinates near the boundary such that $x_0$ is defining function and $\{x^i\}_{i=0}^n$
 are coordinates of $\M.$ Then on $D,$ we have:
\begin{equation}\label{3.1}
g_{ij}=h_{ij}.
\end{equation}
\begin{equation}\label{3.2}
R_{ij}=\frac{n-1}{n-2}(Ric_h)_{ij}+(\frac{1}{2n}S-\frac{1}{2(n-2)}S_h)h_{ij}+\frac{n-1}{2n^2}H^2h_{ij}.
\end{equation}
\begin{equation}\label{3.3}
R_{0i}=-(g^{00})^{-\frac{1}{2}}\frac{n-1}{n}\frac{\partial H}{\partial x_i}-\frac{g^{0j}}{g^{00}}R_{ij}.
\end{equation}
\begin{equation}\label{3.4}
R_{00}=\frac{1}{(g^{00})^2}(g^{0i}g^{0j}R_{ij}+g^{00}(\frac{1}{2}(S-S_h)-\frac{n-1}{2n}H^2)).
\end{equation}
\begin{equation}\label{3.5}
N(R_{0i})=(g^{00})^{-\frac{1}{2}}(-g^{j\beta}\partial_\beta R_{ji}+g^{\eta\beta}\Gamma_{i\beta}^\tau R_{\eta\tau})
\end{equation}
\begin{equation}\label{3.6}
g^{\eta\beta}\partial_\eta (g_{\alpha\beta}-\frac{1}{2}\partial_\alpha g_{\eta\beta})=0
\end{equation}
\begin{equation}\label{3.7}
\partial_kA_{ij}=-\frac{1}{n-1}(g^{00})^{\frac{1}{2}}(R_{0k}+\frac{g^{0j}}{g^{00}}R_{ij})
\end{equation}
where $N=\frac{\nabla x_0}{|\nabla x_0|}=(g^{00})^{-\frac{1}{2}}g^{0\beta}\partial_\beta$ is the unit norm vector on $\M$ and $R_{\alpha\beta},S,H,A$ are Ricci curvature, scalar curvature mean curvature and the second fundamental form respect to $g$ on the boundary.
\end{prop}
For the higher order boundary conditions, we have the following result. It could be obtained by the proof of Proposition 4.1. in \cite{helliwell2008boundary}.
\begin{prop}
Let $(M,g^+)$ be a conformally compact Einstein manifold of dimension $n+1\geq 4$ and $g=\rho^2g^+$ be a $C^{n-1}$ Yamabe compactification. $g|_{\M}=h$ is the boundary metric. Suppose that $(\overline{U}_p,\{y^\beta\}_{\beta=0}^n)$ are any $C^{n}$  structure of $\M\times[0,\epsilon)$ with $y^0|_{\M}=0.$ Then on $\M,$ we have:
\begin{equation}\label{3.8}
\Delta^mRic_{\alpha\beta}=Q(g^{-1},(g^{00})^{-1/2},\partial^{2m+1}g,\partial_t^{2m+2}h)
\end{equation}
for $1\leq m\leq\frac{n+1}{2}-2$ where $Q$ is a polynomial and $\partial_t$ is derivative with respect to coordinates $y^i$ for $i=1,2,...,n.$

\end{prop}
\section{Proof of the Main Theorem}\label{sect4}
Let's begin with the ambient obstruction tensor. We have already showed that for a $n+1$ (even) dimensional manifold, the ambient obstruction tensor
$$
\mathcal{O}_{ij}=\Delta^{\frac{n+1}{2}-2}(P_{ij,k}^{\ \ \ \ k}-P_{k\ \ ,ij}^{\ k})+\mathcal{Q}(Rm),
$$
where  $P_{ij}=\frac{1}{n-1}(R_{ij}-\frac{S}{2n}g_{ij}).$ If the scalar curvature is constant, then
$$P_{ij,k}^{\ \ \ \ k}-P_{k\ \ ,ij}^{\ k}=\frac{1}{n-1}R_{ij,k}^{\ \ \ \ k}=\frac{1}{n-1}\Delta Ric+\Gamma\ast\partial Ric+Rm\ast Rm.$$
From \cite{graham2005ambient} we know that $\mathcal{Q}(Rm)$ denotes quadratic and higher terms in curvature involving fewer derivatives.
 In fact, it was showed earlier in \cite{lopez2018ambient} that the lower order terms is of the following form.
\begin{prop}
For an $(n+1)$-dimensional manifold, the term $\mathcal{Q}(Rm)$ in ambient obstruction tensor is of the formula
$$\mathcal{Q}(Rm)=\sum\limits_{l=2}^{\frac{n+1}{2}}\sum\limits_{i_1+...+i_l=n+1-2l}\nabla^{i_1}(Rm)\ast\nabla^{i_2}(Rm)\cdots\ast\nabla^{i_l}(Rm)$$
Here $A\ast B$ denotes a quadratic term of $A$ and $B$ with  coefficients $g$ and $g^{-1}.$
\end{prop}
\begin{proof}
Let's review the proof of Theorem 2.11 in \cite{graham2005ambient}. Let $n+1\geq 4$ be even and $N$ be an $(n+1)$-dimensional manifold with conformal structure $[g].$  There exists a metric $g^+$ with $x^2g^+$ smooth on $N\times [0,\epsilon)$ such
that $g^+$ has $[g]$ as conformal infinity and $Ric_{g^+} + (n+1)g^+ = O(x^{n-1}).$  In fact, $g^+=x^{-2}(dx^2+g_x)$ for a 1-parameter family $g_x$ of metrics on $N$ with $g_0 = g.$ Then the obstruction tensor is
$$\mathcal{O}=c_n\text{tf}(x^{1-n}(Ric_{g^+} + (n+1)g^+)|_{TN}), \ \ \ \ c_n=\frac{2^{n-1}(\frac{n+1}{2}-1)^2}{n-1}. $$
Let $E=Ric_{g^+} + (n+1)g^+),$ and we need to calculate $E.$ By Gauss Coddazi equation, we obtain that
\begin{equation}\label{4.1}
  2xE_{ij}=-xg''_{ij}+xg^{kl}g'_{ik}g'_{jl}-\frac{x}{2}g^{kl}g'_{kl}g'_{ij}+ng'_{ij}+g^{kl}g'_{kl}g_{ij}+2xRic({g_x})_{ij}
\end{equation}
where $'$  denote $\partial_x$ and $x$ is the normal direction on $N$ as a submanifold of $N\times [0,\epsilon).$
As it is showed in \cite{graham2005ambient}, the derivatives $\partial_x^s(g_x)|_{x=0}$ for $1\leq s\leq n$
are determined inductively by setting $E_{ij} = 0$ and differentiating in \ref{4.1}, and
the obstruction tensor $O_{ij}$ arises when trying to solve for $\partial_x^{n+1}(g_x)|_{x=0}.$ Parity considerations
show that these derivatives vanish for $s$ odd.
\\
\textbf{Step 1,} $g''_{ij}|_{x=0}=-2P_{ij},$ \ \ $R''_{ij}|_{x=0}=P_{ij,k}^{\ \ \ \ k}-P_{k\ ,ij}^{\ k}+Rm\ast Rm.$
\\ Differentiating \ref{4.1} once
gives $g''_{ij}|_{x=0}=-2P_{ij}.$ By the first variation of Ricci curvature,
$$R'_{ij}=\frac{1}{2}({g'_{ik,j}}^{k}+{g'_{jk,i}}^{k}-{g'_{ij,k}}^{k}-{g'_{k}}^{k}_{\ ,ij}).$$
Then $R'_{ij}|_{x=0}=0 $ and
\begin{equation}\label{4.2}
  \begin{aligned}
  R''_{ij}|_{x=0}&=\frac{1}{2}\cdot (-2)({P_{ik,j}}^{k}+{P_{jk,i}}^{k}-{P_{ij,k}}^{k}-{P_{k}}^{k}_{\ ,ij})\\
                 &=-(P_{ik,\ j}^{\ \ \ k}+P_{jk,\ i}^{\ \ \ k}-{P_{ij,k}}^{k}-{P_{k}}^{k}_{\ ,ij})+P\ast Rm\\
                 &=-({P_{k}}^{k}_{\ ,ij}+{P_{k}}^{k}_{\ ,ij}-{P_{ij,k}}^{k}-{P_{k}}^{k}-{P_{k}}^{k}_{\ ,ij})+P\ast Rm\\
                 &={P_{ij,k}}^{k}-{P_{k}}^{k}_{\ ,ij}+P\ast Rm
  \end{aligned}
\end{equation}
Here we use the Bianchi identity ${P_{ik,}}^k={P_{k}}^{k}_{\ ,i}$ and the Ricci formula
$$T_{i_1i_2...i_p,jk}-T_{i_1i_2...i_p,kj}=\sum\limits_{t=1}^pR^s_{i_tjk}T_{i_1i_2...i_{t-1}si_{t+1}...i_p}=Rm\ast T.$$
\textbf{Step 2,} the formula of $R^{(2m)}_{ij}|_{x=0}$ and $g^{(2m+2)}_{ij}|_{x=0}.$
\\ For $m\geq 1,$ let \begin{equation}\label{4.3}
  A_m=\sum\limits_{l=2}^{m+1}\sum\limits_{i_1+...+i_l=2(m+1-l)}\nabla^{i_1}(Rm)\ast\nabla^{i_2}(Rm)\cdots\ast\nabla^{i_l}(Rm).
\end{equation}
For instance, $A_1=Rm\ast Rm$ and $$A_2=\nabla^2(Rm)\ast Rm+\nabla(Rm)\ast\nabla(Rm)+Rm\ast Rm\ast Rm.$$

\textbf{Claim:} $R^{(2m)}_{ij}|_{x=0}=\Delta^{m-1}({P_{ij,k}}^{k}-{P_{k}}^{k}_{\ ,ij})+A_m$ for $1\leq m\leq \frac{n+1}{2}-1.$
\\~\\
 Here we assume that the coefficient of the principal part is $1$ by
multiplying a constant. Now we prove it by induction and it is trival for $m=1.$ If the claim is also correct for $1\leq m< \frac{n+1}{2}-1,$ then differentiating \ref{4.1} $2m+1$ times and setting $x = 0$ gives
\begin{equation}\label{4.4}
  \begin{aligned}
  0=\partial_x^{2m+1}(2xE)|_{x=0}&=(n-1-2m)g^{(2m+2)}_{ij}+g^{kl}g^{(2m+2)}_{kl}g_{ij}
                                \\&+2(2m+1)R^{(2m)}_{ij}+g^{(a)}\ast g^{(b)}\ast g^{(c)}
  \end{aligned}
\end{equation}
Here $a+b+c=2m+2,$ $a,b,c\leq 2m$ and they are all even and hence $g^{(a)}\ast g^{(b)}\ast g^{(c)}=A_m.$ Then
\begin{equation}\label{4.5}
  g^{(2m+2)}_{ij}|_{x=0}=R^{(2m)}_{ij}|_{x=0}+A_m.
\end{equation}
Let $C= g^{(2m+2)}|_{x=0},$ then
\begin{equation}\label{4.6}
  \begin{aligned}
      R^{(2m+2)}_{ij}|_{x=0}&=\frac{1}{2}({C_{ik,j}}^{k}+{C_{jk,i}}^{k}-{C_{ij,k}}^{k}-{C_{k}}^{k}_{\ ,ij})\\
                            &=\frac{1}{2}(C_{ik,\ j}^{\ \ \ k}+C_{jk,\ i}^{\ \ \ k}-{C_{ij,k}}^{k}-{C_{k}}^{k}_{\ ,ij})+C\ast Rm
  \end{aligned}
\end{equation}
For the first term,
\begin{equation}\label{4.7}
  \begin{aligned}
  C_{ik,\ j}^{\ \ \ k}&={[\Delta^{m-1}(P_{ik,t}^{\ \ \ \ t}-P_{s\ ,ik}^{\ s})+A_m]}^k_{\ j}\\
                      &={(P_{ik,t}^{\ \ \ \ t}-P_{s\ ,ik}^{\ s})}_{d_1\ d_2\ ...\ d_{m-1}\ \ j}^{\ d_1\ d_2\ ...\ d_{m-1}k}+A_{m+1}\\
                      &={(P_{ik,t\ \ j}^{\ \ \ \ tk}-P_{s\ ,ik\ j}^{\ s\ \ \ k})}_{d_1\ d_2\ ...\ d_{m-1}}^{\ d_1\ d_2\ ...\ d_{m-1}}+A_{m+1}\\
                      &=\Delta^{m-1}(P_{ik,\ t\ j}^{\ \ \ k\ t}-P_{s\ ,ik\ j}^{\ s\ \ \ k})+A_{m+1}\\
                      &=\Delta^{m-1}(P_{k\ ,it\ j}^{\ k\ \ \ t}-P_{s\ ,ik\ j}^{\ s\ \ \ k})+A_{m+1}\\
                      &=A_{m+1}
    \end{aligned}
\end{equation}
Here we use the Ricci formula on and on. Similarly, $C_{jk,\ i}^{\ \ \ k}=A_{m+1}.$ It is trivial that ${C_{k}}^{k}=A_m,$ which gives ${C_{k}}^{k}_{\ ,ij}=A_{m+1}.$ Let's continue to deal with
\ref{4.6}, and obtain that
$$R^{(2m+2)}_{ij}|_{x=0}=-\frac{1}{2}{C_{ij,k}}^{k}+A_{m+1}=-\frac{1}{2}\Delta^{m}({P_{ij,k}}^{k}-{P_{k}}^{k}_{\ ,ij})+A_{m+1}$$
Hence the claim is also true for $m+1.$
\\
\textbf{Step 3,the ambient obstruction tensor}
\\ Let $m=\frac{n+1}{2}-1$ in \ref{4.4}, we obtain that
\begin{equation}\label{4.8}
  \partial_x^n(2xE_{ij})|_{x=0}=g^{kl}g^{(n+1)}_{kl}g_{ij}+2nR^{(n-1)}_{ij}+g^{(a)}\ast g^{(b)}\ast g^{(c)},
\end{equation}
where $a+b+c=n+1,$ $a,b,c\leq n-1$ and they are all even. Hence $g^{(a)}\ast g^{(b)}\ast g^{(c)}=A_{\frac{n-1}{2}}.$
Since $\text{tf}(g_{ij})=0,$ $\text{tf}(A_{\frac{n-1}{2}})=A_{\frac{n-1}{2}}$ and
$$\text{tf}(R^{(n-1)}_{ij}|_{x=0})=R^{(n-1)}_{ij}-\frac{g^{kl}R^{(n-1)}_{kl}}{n+1}g_{ij}=R^{(n-1)}_{ij}+A_{\frac{n-1}{2}},$$
finally, we could get that
\begin{equation}\label{4.9}
  \begin{aligned}
 \mathcal{O}_{ij}&=\text{tf}(x^{1-n}(Ric_{g^+} + (n+1)g^+)|_{TN})\\
                 &=\partial_x^n(2xE_{ij})|_{x=0}=R^{(n-1)}_{ij}+A_{\frac{n-1}{2}}\\
                 &=\Delta^{\frac{n-1}{2}}({P_{ij,k}}^{k}-{P_{k}}^{k}_{\ ,ij})+A_{\frac{n-1}{2}}
     \end{aligned}
\end{equation}
Here we ignore the coefficient of the principal part and treat all of them as $1.$
\end{proof}

Now we can rewrite the ambient obstruction tensor flat equation as follows:
\begin{equation}\label{4.10}
  \begin{aligned}
\Delta^{\frac{n+1}{2}-2}&(\Delta Ric+\Gamma\ast\partial Ric)=\mathcal{Q}\\
                    &=\sum\limits_{l=2}^{\frac{n+1}{2}}\sum\limits_{i_1+...+i_l=n+1-2l}\nabla^{i_1}(Rm)\ast\nabla^{i_2}(Rm)\cdots\nabla^{i_l}(Rm)
 \end{aligned}
\end{equation}
Here $\mathcal{Q}$ is a polynomial in curvature involving at most $(n-3)$-th derivative with coefficients $g$ and $g^{-1}$, and the degree of $\mathcal{Q}(Rm)$ is $\frac{n+1}{2}.$

\subsection{Regularity of the Yamabe Metric $g$}
\par We have already solved the problem in the case where $n+1=4$ in \cite{jin2019finite}. Let $n+1\geq 6$ and $$u=\Delta^{\frac{n+1}{2}-3}(\Delta Ric+\Gamma\ast\partial Ric).$$ For a $C^{n-1}$ conformally compact Einstein metric $g=\rho^2g^+,$  $\rho\in C^{n-1,\sigma},$  we know that  $Ric\in C^{n-3}(\overline{M})$ in the initial smooth y-coordinates. We observe that from (\ref{2.2})
$$\rho Ric=-(n-1)D^2\rho+[\frac{n(|\nabla\rho|^2-1)}{\rho}-\Delta\rho]g=Q(\partial g,\partial^2\rho)\in C^{n-3,\sigma}(\overline{M},\{y\}).$$
Now let's deal with the metric and curvature in harmonic coordinates $\{x^\beta\}_{\beta=0}^n.$ As $g$ is $C^{n-1},$ $x\in C^{n-1,\alpha}(y), \forall \alpha\in(0,1).$ Then in x-coordinates, $g\in C^{n-2,\lambda}(\overline{M},\{x\})$ and
\begin{equation}\label{4.11}
Ric(\frac{\partial}{\partial x^\alpha},\frac{\partial}{\partial x^\beta})=\frac{\partial y^\gamma}{\partial x^\alpha}\frac{\partial y^\tau}{\partial x^\alpha}Ric(\frac{\partial}{\partial y^\gamma},\frac{\partial}{\partial y^\tau})\in C^{n-3}(\overline{M},\{x\}).
\end{equation}
Then $u$ is continuous in $x-$coordinates. We also have that
\begin{equation}\label{4.12}
\rho Ric(\frac{\partial}{\partial x^\alpha},\frac{\partial}{\partial x^\beta})=\rho\frac{\partial y^\gamma}{\partial x^\alpha}\frac{\partial y^\tau}{\partial x^\alpha}Ric(\frac{\partial}{\partial y^\gamma},\frac{\partial}{\partial y^\tau})\in C^{n-3,\sigma}(\overline{M},\{x\})
\end{equation}
which yield
$$\partial^{n-3}(\rho Ric)\in C^\sigma(\overline{M}).$$
Recall the regularity of $\rho$ and $Ric.$ As a consequence, $\rho\partial^{n-3}(Ric)\in C^\sigma(\overline{M}).$
By Lemma 4.4 in \cite{jin2019finite}, we conclude that
\begin{equation}\label{4.13}
\begin{aligned}
  |\partial^{n-3}(Ric)|_\sigma^{(1-\sigma)}&\leq C(|\rho\partial^{n-3}(Ric)|_\sigma+|\partial^{n-3}(Ric)|_0)\\
                       &\leq C(|\rho Ric|_{n-3+\sigma}+|\rho|_{n-3}|Ric|_{n-4}+|\partial^{n-3}(Ric)|_0)\\
                       &\leq C(n,|\rho Ric|_{n-3+\sigma},|\rho|_{n-3},|Ric|_{n-3}).\\
                       &\leq C(n,|\rho|_{n-1+\sigma},|g|_{n-2,\sigma},|Rm|_{n-3}).
 \end{aligned}
\end{equation}
So $$Ric\in H_{n-3+\sigma}^{(4-n-\sigma)}(\overline{M}).$$
\\ For any $p\in \M,$ let $(\overline{U},\{x^\beta\}_{\beta=0}^n)$ be the local harmonic chart of $p.$ Let $D=\overline{U}\cap \M$ be the boundary portion and $U=\overline{U}\cap M.$ Let $g\in C^{n-1}(\overline{U})$ be the Yamabe compactification.  Assume that $\lambda$ is the minimum eigenvalue of $g_{\alpha\beta}$ in $\overline{U}.$ Then locally, $u\in C^0(\overline{U})$ , $Ric\in H_{n-3+\sigma}^{(4-n-\sigma)}(\overline{U})$ and
\begin{equation}\label{4.14}
  |u|_0\leq C(|Rm|_{n-3},|g|_{n-4}),\ \ \ \ \ |Ric|_{n-3+\sigma,U}^{(4-n-\sigma)}\leq |Rm|_{n-3}+|\partial^{n-3}(Ric)|_{\sigma,U}^{(1-\sigma)}
\end{equation}
\begin{lemm}\label{lemma 4.2}
In harmonic coordinates, $g\in H_{n-1+\sigma}^{(2-n-\sigma)}(\overline{U}).$
\end{lemm}
\begin{proof}
In harmonic charts,
$$\Delta g_{\alpha\beta}=-2R_{\alpha\beta}+Q(g,\partial g).$$
Since $g_{\alpha\beta}\in C^{n-2,\alpha}(\overline{U}),$ then $g_{\alpha\beta}\in C^{n-2,\alpha}(D).$ Let $a=n-1+\sigma,\ b=n-2+\sigma,$ and according to the intermediate Schauder theory, $g_{\alpha\beta}\in H_{n-1+\sigma}^{(2-n-\sigma)}(\overline{U}).$ Furthermore,
\begin{equation}\label{4.15}
|g_{\alpha\beta}|_{n-1+\sigma}^{(2-n-\sigma)}\leq C(n,|g|_{1+\sigma},\overline{U},\lambda)(|g|_0+|g|_{n-2+\sigma,D}+|Ric|_{n-3+\sigma}^{(4-n-\sigma)})
\end{equation}
\end{proof}
Now we have that $g\in  H_{n-1+\sigma}^{(2-n-\sigma)}(\overline{U}),$ so the curvature $Rm\in H_{n-3+\sigma}^{(4-n-\sigma)}(\overline{U}).$ By linear transformation of tensor in coordinate system (similar to (\ref{4.2})), $Rm\in C^{n-3}(\overline{U}).$ Recall $\mathcal{Q}$ in (\ref{4.10}) and we derive that $\mathcal{Q}\in H_\sigma^{(1-\sigma)}(\overline{M})$ from the basic property 4 in section 2.3 and
\begin{equation}\label{4.16}
|\mathcal{Q}|_\sigma^{(1-\sigma)}\leq C(n,\overline{U})|Rm|_{n-3} |Rm|_{n-3+\sigma}^{(4-n-\sigma)}\leq C(n,\overline{U},|Rm|_{n-3},|g|_{n-1+\sigma}^{(2-n-\sigma)})
\end{equation}
Finally, using (\ref{4.13})-(\ref{4.16}), we could obtain that
\begin{equation}\label{4.17}
|\mathcal{Q}|_\sigma^{(1-\sigma)}\leq C(n,\overline{U},\lambda,|g|_{n-2,\sigma},|Rm|_{n-3},|\rho|_{n-1+\sigma})
\end{equation}

Now we conclude that $u\in C^\infty(U)\cap C^0(\overline{U})$ is the solution of the ambient obstruction flat equation
$$\Delta u=\mathcal{Q}  \ \ in \ \ U$$
with the boundary condition
$$u|_D=[\Delta^{\frac{n+1}{2}-2}Ric+\Delta^{\frac{n+1}{2}-3}(\Gamma\ast\partial Ric)]|_D=Q(g^{-1},(g^{00})^{-1/2},\partial^{n-2}g,\partial_t^{n-1}h)$$
This yields $u\in H_{2+\sigma}^{(-\sigma)}(U)$ and
\begin{equation}\label{4.18}
\begin{aligned}
|u|_{2+\sigma,U}^{(-\sigma)}&\leq C(\overline{U},\lambda,|g|_{1+\sigma})(|u|_0+|u|_{\sigma,D}+|\mathcal{Q}|_\sigma^{(2-\sigma)})
                          \\&\leq C(n,\overline{U},\lambda,|g|_{n-2,\sigma},|Rm|_{n-3},|h|_{n-1+\sigma,D},|\rho|_{n-1+\sigma})
                          \\&\leq C(n,\overline{U},\lambda,|g|_{1+\sigma},|Rm|_{n-3},|h|_{n-1+\sigma,D},|\rho|_{n-1+\sigma})
\end{aligned}
\end{equation}
The last inequality holds because in harmonic chart,
$$|g|_{n-2,\sigma}\leq C(n,\overline{U},\lambda,|g|_{1+\sigma})(|Ric|_{n-4,\sigma}+|h|_{n-2,\sigma}).$$
Particularly, $u$ is H\"older continuous, i.e. $u=\Delta^{\frac{n+1}{2}-3}(\Delta Ric+\Gamma\ast\partial Ric)\in C^\sigma(\overline{U}).$ It follows from the classic Schauder theory \cite{gilbarg2015elliptic} with boundary condition (\ref{3.8}) that
$$\Delta Ric+\Gamma\ast\partial Ric\in C^{n-5,\sigma}(\overline{U}),$$
Then using the boundary conditions (\ref{3.2}),(\ref{3.4}) and (\ref{3.5}), we obtain
$R_{\alpha\beta}\in C^{n-3,\sigma}(\overline{U}).$
Finally, as in harmonic chart,
$$\Delta g_{\alpha\beta}=-2R_{\alpha\beta}+Q(g,\partial g)$$
then we could use the boundary conditions (\ref{3.1}),(\ref{3.6}) and (\ref{3.7}) to derive that
$g_{\alpha\beta}\in C^{n-1,\sigma}(\overline{U})$ and
$$|g_{\alpha\beta}|_{n-1+\sigma}\leq C(n,\overline{U},\lambda,|g|_{1+\sigma},|Rm|_{n-3},|h|_{n-1+\sigma,D},|\rho|_{n-1+\sigma})$$
\par Now we look back on the ambient obstruction flat equation (\ref{4.10}) and we have already shown that $\mathcal{Q}\in C^\sigma(\overline{U}).$
 Repeat the steps above (or use the classic Schauder theory ), we could improve the regularity of metric $g$ gradually, and finally $g\in C^{m,\alpha}(\overline{U})$ and
$$|g_{\alpha\beta}|_{m+\alpha}\leq  C(n,\overline{U},\lambda,|g|_{1+\sigma},|Rm|_{n-3},|h|_{m+\alpha,D},|\rho|_{n-1+\sigma}).$$

\subsection{Regularity of the New Structure and Defining Function}
We have already showed that $g\in C^{m,\alpha}$ in harmonic chart, so $\bigcup\limits_{p\in\M}(\overline{U}_p,\{x^\theta\}_{\theta=0}^n)$ form a $C^{m+1,\alpha}$ differential structure of  $\M\times[0,\epsilon)$ for some $\epsilon>0.$

\par If $g=\rho^2g^+$ is a Yamabe compactification, then $\rho\in C^{n-1,\sigma}(\overline{M})$ in harmonic coordinates and $\rho$ is smooth in interior. We only need to study the boundary regularity of the defining function. For any $p\in\M,$ take the harmonic chart $(\overline{U},\{x\})$ of $p$ and let $U=\overline{U}\cap M, \ D=\overline{U}\cap\M.$ We could also assume that $g_{\alpha\alpha}=1, g_{ij}=g_{02}=g_{03}=\cdots=g_{0n}=0 (i\neq j), g_{01}=g_{10}=\delta$ at $p$ where $\delta\in(0,1)$ is sufficiently close to 1. According to (\ref{2.2}) and (\ref{2.3})
$$Ric-\frac{Sg}{n+1}=-(n-1)\frac{D^2\rho}{\rho}+\frac{n-1}{n+1}\frac{\Delta\rho}{\rho}g.$$
Locally, when acting on $(\frac{\partial}{\partial x^0},\frac{\partial}{\partial x^1})$,
\begin{equation}\label{4.19}
\Delta\rho-(n+1)\cdot g_{01}^{-1}\cdot D^2\rho(\frac{\partial}{\partial x^0},\frac{\partial}{\partial x^1})=\frac{n+1}{n-1}\cdot g_{01}^{-1}\cdot\rho(Ric_{01}-\frac{Sg_{01}}{n+1})
\end{equation}
If $1-\delta$ is small enough, the left side of the formula above is an elliptic operator around $p.$ $\rho|_D\equiv 0,$  $\rho\in C^{n-1,\sigma}(\overline{U}).$ In order to improve the $C^{m+1,\alpha}$ regularity of $\rho,$ we require that $\rho(Ric_{01})$ in (\ref{4.19}) is at least $C^{m-1,\alpha}.$  We have used the Bach flat equation to solve this problem in dimension 4 in \cite{jin2019finite}. If the dimension $n+1\geq 6,$ we could consider the ambient obstruction tensor instead of Bach tensor.
\par We use the symbol $Q_i^j$ to denote a polynomial in metric $g$ and defining function $\rho$ involving at most $i$-th derivative of $g$ and $j$-th derivative of $\rho.$
Actually,
$$\Delta(\rho R_{01})=\rho\Delta(R_{01})+R_{01}\Delta\rho+2g(\nabla\rho,\nabla R_{01})=\rho\Delta(R_{01})+Q_3^2.$$
It can be obtained by iterative method that
$$\Delta^k(\rho R_{01})=\rho\Delta^k(R_{01})+Q_{2k+1}^{2k}$$
for $1\leq k\leq \frac{n+1}{2}-1.$   By the ambient obstruction tensor flat equation,
$$\Delta^{\frac{n+1}{2}-1}R_{01}=Q_n^0.$$
 Now let's consider the following elliptic equation:
\begin{equation}\label{4.20}
 \left\{
    \begin{array}{l}
    \Delta^{\frac{n+1}{2}-1}(\rho R_{01})=Q_n^{n-1} \ \ in \ \  \overline{U}
    \\ \Delta^{\frac{n+1}{2}-2}(\rho R_{01})|_D=Q_{n-2}^{n-3}
    \\ ...
    \\ \Delta(\rho R_{01})|_D=Q_2^3
    \\ \rho R_{01}|_D=0
    \end{array}
 \right.
 \end{equation}
Then $\rho(Ric_{01})\in C^{m-1,\alpha}.$ So the defining function $\rho$ is $C^{m+1,\alpha}.$

\bibliographystyle{plain}%

\bibliography{bibfile}

\begin{thebibliography}{10}

\bibitem{anderson2004boundary}
Michael Anderson, Atsushi Katsuda, Yaroslav Kurylev, Matti Lassas, and Michael
  Taylor.
\newblock Boundary regularity for the ricci equation, geometric convergence,
  and gel’fand’s inverse boundary problem.
\newblock {\em Inventiones mathematicae}, 158(2):261--321, 2004.

\bibitem{anderson2008einstein}
Michael~T Anderson.
\newblock Einstein metrics with prescribed conformal infinity on 4-manifolds.
\newblock {\em Geometric and Functional Analysis}, 18(2):305--366, 2008.

\bibitem{besse2007einstein}
Arthur~L Besse.
\newblock {\em Einstein manifolds}.
\newblock Springer Science \& Business Media, 2007.

\bibitem{chang2018compactness1}
Sun-Yung~A Chang and Yuxin Ge.
\newblock Compactness of conformally compact einstein manifolds in dimension 4.
\newblock {\em Advances in Mathematics}, 340:588--652, 2018.

\bibitem{chang2021compactness}
Sun-Yung~A Chang, Yuxin Ge, Xiaoshang Jin, and Jie Qing.
\newblock On compactness conformally compact einstein manifolds and uniqueness
  of graham-lee metrics, iii.
\newblock {\em preprint}, 2021.

\bibitem{chang2020compactness}
Sun-Yung~A Chang, Yuxin Ge, and Jie Qing.
\newblock Compactness of conformally compact einstein 4-manifolds ii.
\newblock {\em Advances in Mathematics}, 373:107325, 2020.

\bibitem{chrusciel2005boundary}
Piotr~T Chru{\'s}ciel, Erwann Delay, John~M Lee, Dale~N Skinner, et~al.
\newblock Boundary regularity of conformally compact einstein metrics.
\newblock {\em Journal of Differential Geometry}, 69(1):111--136, 2005.

\bibitem{deturck1981some}
Dennis~M DeTurck and Jerry~L Kazdan.
\newblock Some regularity theorems in riemannian geometry.
\newblock In {\em Annales scientifiques de l'{\'E}cole Normale Sup{\'e}rieure},
  volume~14, pages 249--260, 1981.

\bibitem{gilbarg1980intermediate}
David Gilbarg and Lars H{\"o}rmander.
\newblock Intermediate schauder estimates.
\newblock {\em Archive for Rational Mechanics and Analysis}, 74(4):297--318,
  1980.

\bibitem{gilbarg2015elliptic}
David Gilbarg and Neil~S Trudinger.
\newblock {\em Elliptic partial differential equations of second order}.
\newblock springer, 2015.

\bibitem{graham1985conformal}
C~FEFFERMAN-CR Graham and C~Fefferman.
\newblock Conformal invariants.
\newblock {\em Elie Cartan et mathematiques d'aujourd'hui, Asterisque, hors
  serie (Societe Mathematique de France, Paris)}, pages 95--116, 1985.

\bibitem{graham2005ambient}
C~Robin Graham and Kengo Hirachi.
\newblock The ambient obstruction tensor and q-curvature.
\newblock {\em AdS/CFT correspondence: Einstein metrics and their conformal
  boundaries}, 8:59--71, 2005.

\bibitem{helliwell2008boundary}
Dylan~William Helliwell.
\newblock Boundary regularity for conformally compact einstein metrics in even
  dimensions.
\newblock {\em Communications in Partial Differential Equations},
  33(5):842--880, 2008.

\bibitem{jin2017extension}
Xiao~Shang Jin.
\newblock An extension of weyl schouten theorem for lipschitz and h 2
  manifolds.
\newblock {\em Acta Mathematica Sinica, English Series}, 33(7):926--932, 2017.

\bibitem{jin2019finite}
Xiaoshang Jin.
\newblock Finite boundary regularity for conformally compact einstein manifolds
  of dimension 4.
\newblock {\em The Journal of Geometric Analysis}, 31:4004--4023, 2021.

\bibitem{lee1994spectrum}
John~M Lee.
\newblock The spectrum of an asymptotically hyperbolic einstein manifold.
\newblock {\em arXiv preprint dg-ga/9409003}, 1994.

\bibitem{lee2006fredholm}
John~M Lee.
\newblock {\em Fredholm operators and Einstein metrics on conformally compact
  manifolds}, volume~13.
\newblock American Mathematical Soc., 2006.

\bibitem{lopez2018ambient}
Christopher Lopez.
\newblock Ambient obstruction flow.
\newblock {\em Transactions of the American Mathematical Society},
  370(6):4111--4145, 2018.

\bibitem{maldacena1999large}
Juan Maldacena.
\newblock The large-n limit of superconformal field theories and supergravity.
\newblock {\em International journal of theoretical physics}, 38(4):1113--1133,
  1999.

\end{thebibliography}
\noindent{Xiaoshang Jin}\\
  School of mathematics and statistics, Huazhong University of science and technology, Wuhan, P.R. China. 430074
 \\Email address: jinxs@hust.edu.cn

\end{document}